\documentclass{amsart}
\usepackage{amsfonts,amssymb,amscd,amsmath,enumerate,verbatim,calc}
\usepackage[all]{xy}

\newcommand{\wrt}{with respect to}

\newcommand{\m}{\mathfrak{m} }

\newcommand{\E}{\mathcal{E}}
\newcommand{\Z}{\mathbb{Z} }
\newcommand{\bP}{\mathbb{\partial}}

\newcommand{\rt}{\rightarrow}

\newcommand{\ov}{\overline}

\newcommand{\F}{\mathcal{F} }
\newcommand{\G}{\mathcal{G} }

\newcommand{\K}{\mathbb{K} }

\theoremstyle{plain}

\newtheorem{thm}{Theorem}

\newtheorem{theorem}{Theorem}[section]

\newtheorem{lemma}[theorem]{Lemma}
\newtheorem{proposition}[theorem]{Proposition}

\theoremstyle{definition}

\newtheorem{s}[theorem]{}
\newtheorem{remark}[theorem]{Remark}
\newtheorem{example}[theorem]{Example}

\theoremstyle{remark}

\numberwithin{equation}{theorem}

\begin{document}

\title[De Rahm]{De Rahm cohomology of local cohomology modules-The graded case}
 \author{Tony~J.~Puthenpurakal}
\date{\today}
\address{Department of Mathematics, Indian Institute of Technology Bombay, Powai, Mumbai 400 076}

\email{tputhen@math.iitb.ac.in}


\begin{abstract}
Let $K$ be a field of characteristic zero, $R = K[X_1,\ldots,X_n]$.  Let $A_n(K)  = K<X_1,\ldots,X_n, \partial_1, \ldots, \partial_n>$  be the $n^{th}$ Weyl algebra over $K$.  We consider the case when $R$ and $A_n(K)$ is graded by giving $\deg X_i = \omega_i $ and $\deg \partial_i = -\omega_i$  for $i =1,\ldots,n$ (here $\omega_i$ are positive integers).  Set $\omega = \sum_{k=1}^{n}\omega_k$. Let $I$ be a graded ideal in $R$. By a result due to Lyubeznik  the local cohomology modules $H^i_I(R)$ are holonomic $A_n(K)$-modules for each $i \geq 0$.
In this article we prove that  the De Rahm cohomology modules $H^*(\bP ; H^*_I(R))$ is concentrated in degree 
$- \omega$, i.e., $H^*(\bP ; H^*_I(R))_j = 0$
for $j \neq - \omega$. As an application when $A = R/(f)$ is an isolated singularity  we relate $H^{n-1}(\bP ; H^1_{(f)}(R)$ to  $H^{n-1}(\partial(f); A)$,  the $(n-1)^{th}$ Koszul cohomology of $A$ \wrt \ $ \partial_1(f),\ldots,\partial_n(f)$.
\end{abstract}

\maketitle

\section*{Introduction}
Let $K$ be a field of characteristic zero and let $R = K[X_1,\ldots,X_n]$.  
We consider $R$ graded with $\deg X_i  = \omega_i$ for $i = 1,\ldots, n$; here $\omega_i$ are positive integers.  Set
$\m = (X_1,\ldots,X_n)$.
Let $I$ be a graded ideal in $R$. The local cohomology modules $H^*_I(R)$ are clearly graded $R$-modules. Let
$A_n(K)  = K<X_1,\ldots,X_n, \partial_1, \ldots, \partial_n>$  be the $n^{th}$ Weyl algebra over $K$. By a result due to Lyubeznik, see \cite{Ly}, the local cohomology modules $H^i_I(R)$ are \textit{holonomic} $A_n(K)$-modules for each $i \geq 0$.  We can consider $A_n(K)$ graded by giving $\deg \partial_i = -\omega_i$ for 
 $i = 1,\ldots, n$.
 
 Let $N$ be a graded left $A_n(K)$ module. Now $\bP = \partial_1,\ldots,\partial_n$ are pairwise commuting $K$-linear maps. So we can consider the De Rahm complex
$K(\bP;N)$. Notice that the De Rahm cohomology modules  $H^*(\bP;N)$ are in general only \textit{graded} $K$-vector spaces. They are finite dimensional if $N$ is holonomic; \cite[Chapter 1, Theorem 6.1]{B}.
 In particular $H^*(\bP;H^*_I(R))$ are finite dimensional graded $K$-vector spaces. 
 
 Our first result is
 \begin{thm}\label{qE}
 Let $I$ be a graded ideal in $R$. Set $ \omega = \sum_{i=1}^{n} \omega_i$. Then the De Rahm cohomology modules $H^*(\partial_1,\ldots,\partial_n ; H^*_I(R))$ is concentrated in degree $-\omega$, i.e., 
$$H^*(\partial_1,\ldots,\partial_n ; H^*_I(R))_j = 0, \quad 
\text{for $j \neq - \omega$}.$$ 
 \end{thm}
 We give an application of Theorem \ref{qE}. Let $f$ be a homogenous polynomial in $R$ with $A = R/(f)$ an isolated singularity, i.e., $A_P$ is regular for all homogeneous prime ideals $P \neq \m$. Let  $H^i(\partial(f); A)$ be the $i^{th}$ Koszul cohomology of $A$ \wrt \ $ \partial_1(f),\ldots,\partial_n(f)$. We show
\begin{thm}\label{thm2}
(with hypotheses as above) There exists a filtration $\F = \{ \F_\nu \}_{\nu\geq 0}$ consisting of $K$-subspaces of $H^{n-1}(\bP; H^1_{(f)}(R))$ with $\F_\nu = H^{n-1}(\bP; H^1_{(f)}(R))$ for $\nu \gg 0$, $\F_\nu \supseteq F_{\nu-1}$ and $\F_0 = 0$ and injective $K$-linear maps
\[
\eta_\nu \colon \frac{\F_\nu}{\F_{\nu-1}} \longrightarrow H^{n-1}\left(\partial(f); A\right)_{(\nu+1)\deg f - \omega}.
\]
\end{thm}
The techniques in this theorem is generalized in \cite{PR2} to show that $H^i(\bP; H^1_{(f)}(R)) = 0$ for $i < n-1$.
There is no software to compute De-Rahm cohomology of a $A_n(K)$-module $M$. As an application of Theorem 2 we prove
\begin{example}
Let $R = K[X_1,\ldots,X_n]$ and let $f = X_1^2 + X_2^2 + \ldots + X_{n-1}^{2} + X_n^m$ with $m \geq 2$. 
Then
\begin{enumerate}
\item
if $m$ is odd then
$H^{n-1}(\partial; H^1_{(f)}(R)) = 0$.
\item
if $m$ is even then
\begin{enumerate}
\item
If $n$ is odd then $H^{n-1}(\partial; H^1_{(f)}(R)) = 0$.
\item
If $n$ is even then $\dim_K H^{n-1}(\partial; H^1_{(f)}(R)) \leq 1$.
\end{enumerate}
\end{enumerate}
\end{example}
We now describe in brief the contents of the paper. In section one we discuss a few preliminaries that we need. In section two we introduce the concept of generalized Eulerian modules. In section three we give a proof of Theorem 1. In section four we give an outline of proof of Theorem 2. In section five we prove Theorem 2. In section six we give a proof of Example 0.1.
\section{Preliminaries}
 In  this section  we discuss a few preliminary results that we need.
\begin{remark}
Although all the results are stated for De-Rahm cohomology of a $A_n(K)$-module $M$, we will actually work with
De-Rahm homology. Note that $H_i(\bP, M) = H^{n-i}(\bP, M)$ for any $A_n(K)$-module $M$. Let $S = K[\partial_1,\ldots,\partial_n]$. Consider it as a subring of $A_n(K)$. Then note that $H_i(\bP, M)$ is the $i^{th}$ Koszul homology module of $M$ with respect to $\bP$.
\end{remark} 
\s \label{mod-1}
Let $M$ be a holonomic $A_n(K)$-module. Then for $i = 0,1$ the De-Rahm  homology modules $H_i(\partial_n,M)$ are holonomic $A_{n-1}(K)$-modules, see \cite[1.6.2]{B}. 
 
The following result is well-known.
  \begin{lemma}\label{Koszul}
  Let $\bP = \partial_r, \partial_{r+1},\ldots,\partial_{n}$ and $\bP^\prime = \partial_{r+1},\ldots,\partial_{n}$. Let $M$ be a left $A_{n}(K)$-module. For each $i \geq 0$ there exist an exact sequence
  \[
  0 \rt H_0( \partial_r ; H_i(\bP^\prime;M)) \rt H_i(\bP; M) \rt H_1(\partial_r ; H_{i-1}(\bP^\prime ; M)) \rt 0.
  \]
  \end{lemma}

\section{generlized Eulerian modules}
Consider the Eulerain operator
\[
\E_n = \omega_1X_1\partial_1 + \omega_2X_2\partial_2 + \cdots + \omega_nX_n\partial_n.
\]
If $r \in R$ is homogenous then recall that $\E_n r = (\deg r)\cdot r$. Note that degree of $\E_n$ is zero.

 Let $M$ be a graded $A_n(K)$ module.
If $m$ is homogenous, we set $|m| = \deg m$
 We say
that $M$ is \textit{Eulerian} $A_n(K)$-module if $\E_n m = |m|\cdot m$ for each homogenous $m \in M$. This notion was discovered by  Ma and Zhang, see  the very nice paper \cite{MZ}. They prove that local cohomology modules $H^*_I(R)$ are Eulerian $A_n(K)$-modules, see \cite[5.3]{MZ}. Infact they prove this  when $R$ is standard graded. The same proof can be adapted to prove the general case. 

It can be easily seen that if $M$ is Eulerian $A_n(K)$ module then so are each graded submodule  and graded quotient of $M$. However extensions of Eulerian modules need not be Eulerian, see \cite[3.5]{MZ}. To rectify this we introduce the following notion: A graded $A_n(K)$-module $M$ is said to be\textit{ generalized Eulerian} if for a homogenous element 
$m$ of $M$ there exists a positive integer $a$, (here $a$ may depend on $m$) such that
\[
\left(\E_n - |m|\right)^a m = 0.
\]
We now prove that the class of generalized Eulerian modules is closed under extensions.
\begin{proposition}\label{exact}
Let $0 \rt M_1 \xrightarrow{\alpha_1} M_2 \xrightarrow{\alpha_2} M_3 \rt 0$ be a short exact sequence of graded $A_n(K)$-modules.
Then the following are equivalent:
\begin{enumerate}[\rm (1)]
\item
$M_2$ is generalized Eulerian.
\item
$M_1$ and $M_3$ is generalized Eulerian.
\end{enumerate}
\end{proposition}
\begin{proof}
The assertion $(1) \implies (2)$ is clear.
We prove $(2) \implies (1)$. Let $m \in  M_2$ be homogenous. As $M_3$ is  generalized Eulerian we have
\[
(\E_n - |m|)^b\alpha_2(m) = 0 \quad \text{for some} \ b \geq 1.
\]
 Set $v_2 = (\E_n - |m|)^bm \in M_2$.  As $\alpha_2$ is $A_n(K)$-linear we get $\alpha_2(v_2) = 0.$ So $v_2 = \alpha_1(v_1)$ for some $v_1 \in M_1$. Note $\deg v_1 = \deg v_2 = |m|$. As $M_1$ is  generalized Eulerian
 we have that  
 \[
(\E_n - |m|)^a v_1 = 0 \quad \text{for some} \ a \geq 1.
\]
As $\alpha_1$ is $A_n(K)$-linear we get $(\E_n - |m|)^a v_2 = 0$. It follows that
\[
(\E_n - |m|)^{a+b}m = 0.
\]
\end{proof}
If $M$ is a graded $A_n(K)$-module then for $l \in \Z$ the module $M(l)$ denotes the shift of $M$ by $l$, i.e., $M(l)_n = M_{n+l}$ for all $n \in \Z$.
The following result was proved for Eulerian $A_n(K)$-modules in \cite[2.4]{MZ}.
\begin{proposition}
Let $M$ be a non-zero generalized Eulerian $A_n(K)$-module. Then  for $l \neq 0$ the module $M(l)$ is NOT a generalized Eulerian $A_n(K)$-module.
\end{proposition}
\begin{proof}
Suppose $M(l)$ is  a generalized Eulerian $A_n(K)$-module for some $l  \neq 0$.
Let $m \in M$ be homogenous of degree $r$ and \textit{non-zero}. As $M$ is generalized Eulerain $A_n(K)$-module we have that
\[
(\E_n - r)^am = 0 \quad \text{for some} \ a \geq 1.
\]
We may assume $(\E_n - r)^{a-1}m \neq 0$. Now $m \in M(l)_{r-l}$. As $M(l)$ is generalized Eulerain  we get that
\[
(\E_n - r + l)^bm = 0 \quad \text{for some} \ b \geq 1.
\]
Notice
\[
0 = (\E_n - r + l)^bm  = \left( l^b + \sum_{i = 1}^{b}\binom{b}{i} l^{b-i}(\E_n -r)^{i}    \right)m.
\]
Multiply on the left by $(\E_n - r)^{a-1}$. We obtain
\[
l^b(\E_n - r)^{a-1}m = 0
\]
As $l \neq 0$ we get $(\E_n - r)^{a-1}m = 0$ a contradiction.
\end{proof}

\section{Proof of Theorem \ref{qE}}
In this section we prove Theorem \ref{qE}. Notice that $H^i_I(R)$ are Eulerian $A_n(K)$-module for all $i \geq 0$. Hence Theorem \ref{qE} follows from the following more general result.
\begin{theorem}\label{gEv}
Let $M$ be a generalized Eulerian $A_n(K)$-module. Then $H_i(\partial; M)$ is concentrated in degree
$-\omega = -\sum_{k =1}^{n} \omega_k$.
\end{theorem}
Before proving \ref{gEv} we need to prove a few preliminary results.
\begin{proposition}\label{base}
Let $M$ be a generalized Eulerian $A_n(K)$-module. Then for $i = 0,1$, the $A_{n-1}(K)$-modules $H_i(\partial_n ; M)(-\omega_n)$ are 
generalized Eulerian.
\end{proposition} 
\begin{proof}
Clearly $H_i(\partial_n ; M)(-\omega_n)$ are 
 $A_{n-1}(K)$-modules for $i = 0,1$.
We have an exact sequence of $A_{n-1}(K)$-modules
\[
0 \rt H_1(\partial_n;M) \rt M(\omega_n) \xrightarrow{\partial_n} M \rt H_0(\partial_n; M) \rt 0.
\]
Note that $H_1(\partial_n ; M)(-\omega_n) \subset M$. Let $\xi \in H_1(\partial_n ; M)(-\omega_n)$ be homogeneous. As 
$M$ is generalized Eulerian we have that
\[
(\E_n - |\xi|)^a\xi = 0 \quad \text{for some} \ a \geq 1.
\]
Notice $\E_n = \E_{n-1} + \omega_nX_n\partial_n$. Also note that $X_n\partial_n$ commutes with $\E_{n-1}$. So
\[
0 = (\E_{n-1} - |\xi| + \omega_nX_n\partial_n)^a\xi  = \left( (\E_{n-1} - |\xi|)^a + (*) X_n\partial_n\right)\xi.
\]
As $\partial_n \xi = 0$ we get that $(\E_{n-1} - |\xi|)^a\xi = 0$.
It follows that $H_1(\partial_n ; M)(-\omega_n)$ is a 
generalized Eulerian $A_{n-1}(K)$-module.

Let $\xi \in H_0(\partial_n ; M)(-\omega_n)$ be homogenous of degree $r$. Then
$\xi = \alpha + \partial_n M$ where $\alpha \in M_{r-\omega_n}$. As $M$ is generalized Eulerian we get that 
\[
(\E_n - r + \omega_n)^a\alpha = 0 \quad \text{for some} \ a \geq 1.
\]
Notice $\E_n = \E_{n-1} + \omega_n X_n \partial_n = \E_{n-1} + \omega_n \partial_nX_n - \omega_n$. So
$\E_n - r + \omega_n = \E_{n-1} - r + \omega_n\partial_n X_n$. Notice that $\partial_nX_n$ commutes with $\E_{n-1}$.
Thus
\[
0 = (\E_{n-1} - r + \omega_n\partial_n X_n)^a \alpha = (\E_{n-1} - r)^a \alpha + \partial_n\cdot* \alpha.
\]
Going mod $\partial_n M$ we get
\[
(\E_{n-1} - r)^a\xi = 0. 
\]
It follows that $H_0(\partial_n ; M)(-\omega_n)$ is a 
generalized Eulerian $A_{n-1}(K)$-module.
\end{proof}
\begin{remark}
If $M$ is  Eulerian then the same proof shows that $H_i(\partial_n ; M)(-\omega_n)$ is a 
 Eulerian $A_{n-1}(K)$-module for $i = 0,1$. However as the proof of the following theorem shows that
we can only prove that $H_1(\partial_{n-1},\partial_n ; M)(-\omega_{n-1} - \omega_n)$
is  a generalized Eulerian $A_{n-1}(K)$-module.
\end{remark}

We now show that
\begin{proposition}\label{base-2}
Let $M$ be a generalized Eulerian $A_n(K)$-module. 
Let \\ $\partial = \partial_i,\partial_{i+1},\ldots,\partial_n$; here $i \geq 2$. Then for each $j \geq 0$ the De Rahm homology module
\[
H_j(\partial;M)(-\sum_{k=i}^{n}\omega_k)
\]
is a 
generalized Eulerian $A_{i-1}(K)$-module.
\end{proposition} 
\begin{proof}
We prove this result by descending induction on $i$.
For $i = n$ the result holds by Proposition \ref{base}. Set $\partial^\prime = \partial_{i+1},\ldots,\partial_n$.
By induction hypothesis $H_j( \partial^\prime;M)(-\sum_{k= i+1}^{n}\omega_k)$ is generalized Eulerian $A_i(K)$-module.
By Proposition \ref{base} again we get that for $l = 0,1$ and for each $j \geq 0$
\[
H_l\left(\partial_i; H_j(\partial^\prime;M)(-\sum_{k= i+1}^{n}\omega_k)\right)(-\omega_i) = H_l\left(\partial_i; H_j(\partial^\prime;M)\right)(-\sum_{k= i}^{n}\omega_k)
\]
is generalized Eulerian. By \ref{Koszul} we have exact sequence
\[
0 \rt H_0(\partial_i; H_j(\partial^\prime;M)) \rt H_j(\partial;M) \rt H_1(\partial_i; H_{j-1}(\partial^\prime;M)) \rt 0.
\]
The modules at the left and right end become generalized Eulerian after shifting by $ -\sum_{k= i}^{n}\omega_k$. By \ref{exact} it follows that
for each $j \geq 0$ the De Rahm homology module
\[
H_j(\partial;M)(-\sum_{k=i}^{n}\omega_k)
\]
is a 
generalized Eulerian $A_{i-1}(K)$-module.
\end{proof}
We now consider the  case when $n = 1$.
\begin{proposition}\label{n1}
Let $M$ be a generalized Eulerian $A_1(K)$-module. Then for $l = 0,1$ the modules $H_l(\partial_1;M)$ is concentrated in degree $-\omega_1$.
\end{proposition}
\begin{proof}
We have an exact sequence of  $K$-vector spaces
\[
0 \rt H_1(\partial_1;M) \rt M(\omega_1) \xrightarrow{\partial_1} M \rt H_0(\partial_1; M) \rt 0.
\]
Let $\xi \in H_1(\partial_1;M)(-\omega_1)$ be homogenous and non-zero. As $\xi \in M$ we have that
\[
(\omega_1X_1\partial_1 - |\xi|)^a \xi = 0 \ \text{for some} \ a \geq 1.
\]
Notice $(\omega_1X_1\partial_1 - |\xi|)^a = (*)\partial_1 + (-1)^a|\xi|^a$.
As $\partial_1\xi = 0$ we get
$(-1)^a|\xi|^a \xi = 0$. As $\xi \neq 0$ we get that $|\xi| =0$. It follows that $H_1(\partial_1;M)$ is concentrated in degree $-\omega_1$.

Let $\xi \in H_0(\partial_1,M)$ be non-zero and homogeneous of degree $r$. Let $\xi = \alpha + \partial_1M$ where
$\alpha \in M_r$. As $M$ is generalized Eulerian we get that
\[
(\omega_1X_1\partial_1 - r)^a\alpha = 0 \ \text{for some} \ a \geq 1.
\]
Notice $\omega_1X_1\partial_1 = \omega_1\partial_1X_1 - \omega_1$.
So we have
\[
0 = ( \omega_1\partial_1X_1 -(r+\omega_1))^a\alpha = \left(\partial_1 * + (-1)^a(r+\omega_1)^a \right)\alpha.
\]
In $M/\partial_1M$ we have $(-1)^a(r+\omega_1)^a\xi = 0$. As $\xi \neq 0$ we get that $r = -\omega_1$.
It follows that $H_0(\partial_1;M)$ is concentrated in degree $-\omega_1$.
\end{proof}
We now give
\begin{proof}[Proof of Theorem \ref{gEv}]
Set $\partial^\prime = \partial_2,\ldots,\partial_n$. By \ref{base-2} $N_j = H_j(\partial^\prime; M)(-\sum_{k =2}^{n} \omega_k)$ is generalized Eulerian $A_1(K)$-module, for each $j \geq 0$. We use exact sequence in \ref{Koszul} and shift it by $-\sum_{k =2}^{n} \omega_k$ to obtain an exact sequence
\[
0 \rt H_0(\partial_1,N_j) \rt H_j(\partial; M)(-\sum_{k =2}^{n} \omega_k) \rt H_1(\partial_1,N_{j-1}) \rt 0,
\]
for each $j \geq 0$. By Proposition \ref{n1} the modules on the left and right of the above exact sequence is concentrated in degree $-\omega_1$. It follows that for each $j \geq 0$ the $K$-vector space $H_j(\partial; M)$ is concentrated in degree $-\omega = -\sum_{k =1}^{n} \omega_k$.
\end{proof}

\section{Outline of proof of Theorem 2} 
The proof of Theorem 2 is a bit long and has a lot of technical details.
For the convenience of the reader we give an outline of the proof.

\s By \cite[Lemma 2.7]{P} we have $H_1(\bP, R_f) \cong H_1(\bP, H^1_{(f)}(R))$. So it is sufficient to work with  $H_1(\bP, R_f)$ in order to prove Theorem 2. We consider elements of $R^m_f$ as column-vectors. For $x \in R^m_f$ we write it as
$x = (x_1,\ldots,x_m)^\prime$; here $\prime$ indicates transpose.

\s Let $\xi \in R_f^m \setminus R^m$.  The element $(a_1/f^i,a_2/f^i,\ldots,a_m/f^i)^\prime$, with $a_j \in R$ for all $j$, is said to be a \textit{normal form} of $\xi$ if
\begin{enumerate}
\item
$\xi = (a_1/f^i,a_2/f^i,\ldots,a_m/f^i)^\prime$.
\item
$f$ does not divide $a_j$ for some $j$.
\item
$i\geq 1$.
\end{enumerate}

It can be easily shown that normal form of $\xi$ exists and is unique; see Proposition \ref{normal}.
Let $(a_1/f^i,a_2/f^i,\ldots,a_m/f^i)^\prime$ be the normal form of $\xi$. Set $L(\xi) = i$. Notice $L(\xi) \geq 1$. 

\s \textit{Construction of a function} $\theta \colon Z_1(\bP,R_f)\setminus R^n \rt H_1(\partial(f);A)$.

Let $\xi \in  Z_1(\bP,R_f)\setminus R^n$. Let $(a_1/f^i,a_2/f^i,\ldots,a_n/f^i)^\prime$ be the normal form of $\xi$. So we have $\sum_{j=1}^n \partial/\partial X_j (a_j/f^i) = 0$. So we have
\[
\frac{1}{f^i}\left( \sum_{j=1}^{n}\frac{\partial a_j}{\partial X_j}   \right)  - \frac{i}{f^{i+1}}\left( \sum_{j=1}^{n}a_j\frac{\partial f}{\partial X_j}   \right) = 0.
\]
It follows that 
$$f  \ \text{ divides } \ \sum_{j=1}^{n}a_j\frac{\partial f}{\partial X_j}.$$
So $(\ov{a_1},\ldots,\ov{a_n})^\prime \in Z_1(\partial(f);A)$.
We set
$$ \theta(\xi) = [(\ov{a_1},\ldots,\ov{a_n})^\prime] \in H_1(\partial(f);A). $$

\begin{remark}
It can be shown that if $\xi \in Z_1(\bP,R_f)_{-\omega}$ is non-zero then $\xi \notin R^n$, see \ref{In-Rf}. If $L(\xi) = i$ then by \ref{comput-degrees} we have 
$$ \theta(\xi) \in H_1(\partial(f);A)_{(i+1)\deg f - \omega}.$$
\end{remark}
The next result uses the fact that $A$ is an isolated singularity.
\begin{proposition}\label{B}
If $\xi \in  B_1(\bP,R_f)_{-\omega}$ is non-zero then $ \theta(\xi) = 0$.
\end{proposition}

\s Let $\xi \in  R_f^m$. We define $L(f)$ as follows.

\textit{Case 1:} $\xi \in R_f^m\setminus R^m$. 

Let $(a_1/f^i,a_2/f^i,\ldots,a_m/f^i)^\prime$ be the normal form of $\xi$. Set $L(\xi) = i$. Notice $L(\xi) \geq 1$ in this case.

\textit{Case 2:} $\xi \in R^m\setminus \{0\}$. 

Set $L(\xi) = 0$.

\textit{Case 3:} $\xi = 0$. 

Set $L(\xi) = -\infty$.

The following properties of the function $L$ can be easily verified.
\begin{proposition}\label{prop-l}
(with hypotheses as above) Let $\xi,\xi_1,\xi_2 \in R_f^m$ and $\alpha, \alpha_1,\alpha_2 \in K$.
\begin{enumerate}[\rm(1)]
\item
If $L(\xi_1) < L(\xi_2)$ then $L(\xi_1 + \xi_2) = L(\xi_2)$.
\item
 If $L(\xi_1) = L(\xi_2)$ then $L(\xi_1 + \xi_2) \leq L(\xi_2)$.
\item
$L(\xi_1 + \xi_2) \leq \max\{ L(\xi_1) , L(\xi_2) \}.$
\item
 If $\alpha \in K^*$ then $L(\alpha \xi ) = L(\xi)$.
\item
$L(\alpha \xi ) \leq L(\xi)$ for all $\alpha \in K$.
\item
$L(\alpha_1\xi_1 + \alpha_2\xi_2) \leq \max\{ L(\xi_1) , L(\xi_2) \}.$
\item
Let $\xi_1,\ldots,\xi_r \in R^m_f$ and let $\alpha_1,\ldots,\alpha_r \in K$. Then
$$L\left(\sum_{j=1}^{r}\alpha_j\xi_j \right) \leq \max \{ L(\xi_1), L(\xi_2),\ldots,L(\xi_r) \}.$$
\end{enumerate}
\end{proposition} 

\s We now use the fact that $H_1(\bP,R_f)$ is concentrated in degree $-\omega = - \sum_{k=1}^{n}\omega_k$.
So 
\[
H_1(\bP,R_f) = H_1(\bP,R_f)_{-\omega} = \frac{Z_1(\bP,R_f))_{-\omega}}{B_1(\bP,R_f)_{-\omega}}.
\]
Let $x \in H_1(\bP,R_f)$ be non-zero. Define
\[
L(x) = \min \{ L(\xi) \mid x = [\xi], \text{where} \ \xi \in Z_1(\bP,R_f)_{-\omega}. \}.
\]
It can be shown that $L(x)\geq 1$.
If $x = 0$ then set
\[
L(0) = -\infty.
\]

 We now define a function
\begin{align*}
\widetilde{\theta} \colon H_1(\bP,R_f) &\rt H_1(\partial(f);A) \\
       x &\mapsto \begin{cases} \theta(\xi) & \text{if} \ x\neq 0, x = [\xi], \text{and} \ L(x) = L(\xi),\\ 0 & \text{if} \ x = 0. \end{cases}
\end{align*}
It can be shown that $\widetilde{\theta}(x)$ is independent of choice of $\xi$; see Proposition \ref{indep-1}. Also note that if $L(x) = i$ then $\widetilde{\theta}(x) \in H_1(\partial(f);A)_{(i+1)\deg f - \omega}$.

\s We now construct a filtration $\F = \{\F_\nu\}_{\nu\geq 0}$ of $H_1(\bP,R_f)$. Set
\[
\F_\nu = \{ x \in H_1(\bP,R_f)\mid L(x) \leq \nu \}.
\]
In the next section we prove
\begin{proposition}
\begin{enumerate}[\rm (1)]
\item
$\F_\nu$ is a $K$ subspace of $H_1(\bP,R_f)$.
\item
$\F_{\nu} \supseteq \F_{\nu-1}$ for all $\nu \geq 1$.
\item
$\F_\nu = H_1(\bP,R_f)$ for all $\nu \gg 0$.
\item
$\F_0 = 0$.
\end{enumerate}
\end{proposition}
Let $\G = \bigoplus_{\nu \geq 1} \F_{\nu}/\F_{\nu-1}$. For $\nu \geq 1$ we define
\begin{align*}
\eta_{\nu} \colon \frac{\F_{\nu}}{\F_{\nu-1}} &\rt H_1(\partial(f);A)_{(\nu + 1)\deg f - \omega}\\
                 \xi &\mapsto \begin{cases}0, & \text{if} \ \xi = 0; \\ \widetilde{\theta}(x), & \text{if} \ \xi = x + \F_{\nu -1} \ \text{is non-zero}.                                                 \end{cases}
\end{align*}
It can be shown that $\eta_{\nu}(\xi)$ is independent of choice of $x$; see \ref{indep-2}.
Finally we prove the following result.

\begin{theorem}(with Notation as above).  For all $\nu \geq 1$,
\begin{enumerate}[\rm (1)]
\item
$\eta_{\nu}$ is $K$-linear.
\item
$\eta_{\nu}$ is injective.
\end{enumerate}

\end{theorem}

\section{Proof of Theorem 2}
In this section we give a proof of Theorem 2 with all details. The reader is advised to read the previous section before reading this section.

We first prove
\begin{proposition}\label{normal}
Let $\xi \in R^m_f\setminus R^m$. Then normal form of $\xi$ exists and is unique.
\end{proposition}
\begin{proof}
\textit{Existence:} Let $\xi \in R^m_f\setminus R^m$. Let $\xi = (b_1/f^{i_1},b_2/f^{i_2},\cdots, b_m/f^{i_m})^\prime$
with $f \nmid b_j$ if $b_j \neq 0$. Note $i_j \leq 0$ is possible. 
 Let 
$$ i_r = \max \{ i_j \mid i_j \geq 1 \ \text{and} \ b_j \neq 0 \}.$$
Notice $i_r \geq 1$. Then
\[
\xi = (\frac{b_1f^{i_r-i_1}}{f^{i_r}}, \frac{b_2f^{i_r-i_2}}{f^{i_r}}, \cdots, \frac{b_mf^{i_r-i_m}}{f^{i_r}})^\prime
\]
Note that $f \nmid b_r$. Thus the expression above is a normal form of $\xi$.

\textit{Uniqueness:} Let $(a_1/f^i,\cdots,a_m/f^i)^\prime$ and $(b_1/f^r,\cdots,b_m/f^r)^\prime$ be two normal forms of $\xi$. We first assert that $i<r$ is not possible. For if this holds then note as
$a_j/f^i = b_j/f^r$ we get $b_j = a_jf^{r-i}$. So $f | b_j$ for all $j$; a contradiction. 

A similar argument shows that $i> r$ is not possible. So $i =r$. Thus $a_j = b_j$ for all $j$. Thus the normal form of $\xi$ is unique.
\end{proof}

\s \label{In-Rf} Let $\xi \in Z_1(\bP,R_f)_{-\omega}$ be non-zero. Let $\xi = (\xi_1,\ldots,\xi_n)^\prime$. Note that
\[
\xi \in \left( R_f(\omega_1)\oplus R_f(\omega_2) \oplus \cdots \oplus R_f(\omega_n)\right)_{-\omega}
\] 
It follows that 
$$\xi_j \in  (R_f)_{-\sum_{k \neq j}\omega_k}.$$
It follows that $\xi \in R^{n}_f \setminus R^n$. 

\s \label{comput-degrees} Let $(a_1/f^i,\cdots,a_n/f^i)^\prime$  be the normal form of $\xi$. Then
\[
\deg a_j = i\deg f -\sum_{k \neq j}\omega_k.
\]
In particular going mod $f$ we get
$$ \overline{a_j} \in A\left(-\deg f + \omega_j \right)_{(i+1)\deg f - \omega }.$$
Notice 
$ \deg \partial{f}/\partial X_j  = \deg f - \omega_j$. It follows that
\[
(\ov{a_1},\cdots,\ov{a_n})^\prime \in Z_1(\partial(f);A)_{(i+1)\deg f - \omega }.
\]
Thus $\theta(\xi) \in H_1(\partial(f);A)_{(i+1)\deg f - \omega }.$ 

\s Let  $\K = \K(\partial;R_f)$ be the De Rahm complex on $R_f$ written homologically. So
\[
\K = \cdots \rt \K_3\xrightarrow{\phi_3} \K_2\xrightarrow{\phi_2} \K_1\xrightarrow{\phi_1} \K_0 \rt 0.
\]
Here $\K_0 = R_f$, $\K_1 = \bigoplus_{k = 1}^{n} R_f(\omega_k)$,
\[
\K_2 = \bigoplus_{1 \leq i < j \leq n}R_f(\omega_i + \omega_j) \quad \text{and} \ \K_3 = \bigoplus_{1 \leq i < j <l\leq n}R_f(\omega_i + \omega_j + \omega_l).
\]
Let $\K^\prime = \K(\partial(f);A)$ be the Koszul complex on $A$ \wrt \ 

\noindent $\partial f/\partial X_1,\ldots, \partial f/\partial X_n$. So
\[
\K^\prime = \cdots \rt \K_3^\prime\xrightarrow{\psi_3} \K_2^\prime\xrightarrow{\psi_2} \K_1^\prime\xrightarrow{\psi_1} \K_0^\prime \rt 0.
\]
Here $\K_0^\prime = A$, $\K^\prime_1 = \bigoplus_{k = 1}^{n}A(-\deg f + \omega_k)$,
\[
\K_2^\prime = \bigoplus_{1 \leq i < j \leq n}A(-2\deg f + \omega_i + \omega_j) \quad \text{and} \ \K_3^\prime = \bigoplus_{1 \leq i < j <l\leq n}A(-3\deg f + \omega_i + \omega_j + \omega_l).
\]

We now give
\begin{proof}[proof of Proposition \ref{B}]
Let $u \in B_1(\partial ; R_f)_{-\omega}$ be non-zero. Let $\xi \in (\K_2)_{-\omega}$ be homogeneous with $\phi_2(\xi) = u$.
Let $\xi = (\xi_{ij} \mid 1\leq i < j \leq n)^\prime$. Notice that
\[
\xi_{ij} \in R_f(\omega_i + \omega_j)_{-\omega} = (R_f)_{-\sum_{k\neq i,j} \omega_k}.
\]
It follows that $\xi \in R_f^{\binom{n}{2}}\setminus R^{\binom{n}{2}}$. Set
\[
c = \min\{j \mid j= L(\xi)\ \text{where} \ \phi_2(\xi) = u \ \text{and}\ \xi \in (\K_2)_{-\omega} \ \text{is homogeneous} \}.
\]
Notice $c \geq 1$. Let $\xi \in (\K_2)_{-\omega}$ be such that $L(\xi) = c$ and $\phi_2(\xi) = u$.
Let $(b_{ij}/f^c  \mid 1\leq i < j \leq n)^\prime$ be the normal form of $\xi$. Let $u = (u_1,\ldots,u_n)^\prime$.
Then for $l = 1,\ldots,n$ 
\[
u_l = \sum_{i<l}\frac{\partial}{\partial X_i}\left(\frac{b_{il}}{f^c}\right) - \sum_{j>l}\frac{\partial}{\partial X_j}\left(\frac{b_{lj}}{f^c}\right)
\]
So 
$$ u_l = \frac{f}{f^{c+1}}\left(\sum_{i<l}\frac{\partial(b_{il})}{\partial X_i} - \sum_{j>l}\frac{\partial(b_{lj})}{\partial X_j}\right) + \frac{c}{f^{c+1}}\left(-\sum_{i<l}b_{il}\frac{\partial f}{\partial X_i}
+ \sum_{j>l}b_{lj}\frac{\partial f}{\partial X_j} \right)$$

Set
\[
v_l = c\left(-\sum_{i<l}b_{il}\frac{\partial f}{\partial X_i}
+ \sum_{j>l}b_{lj}\frac{\partial f}{\partial X_j} \right)
\]
Therefore
\[
u_l = \frac{f* + v_l}{f^{c+1}}.
\]
\textit{Claim:} $f\nmid v_l$ for some $l$.

First assume the claim. Then $((f* +v_1)/f^{c+1},\ldots,(f* + v_n)/f^{c+1})^\prime$ is normal form of $u$.
So 
\[
\theta(u) = [ (\ov{v_1},\ldots,\ov{v_n})^\prime] =  [\psi_2(-c\ov{b})] = 0.
\]

We now prove our claim. Suppose if possible $f | v_l$ for all $l$.
Then
\[
\psi_2(-c\ov{b}) = (\ov{v_1},\ldots,\ov{v_l})^\prime = 0.
\]
So $-cb \in Z_2(\partial(f);A)$. As $H_2(\partial(f); A) = 0$ we get
$-cb \in B_2(\partial(f);A)$.
Thus $-c\ov{b} = \psi_3(\ov{\gamma})$. Here
\[
\gamma = (\gamma_{ijl} \mid 1\leq i<j<l \leq n )^\prime.
\]
So
\begin{equation}\label{bb}
-cb_{ij} = \sum_{k < i <j}\gamma_{kij}\frac{\partial f}{\partial X_k} - \sum_{i < k <j}\gamma_{ikj}\frac{\partial f}{\partial X_k} + \sum_{i < j < k}\gamma_{ijk}\frac{\partial f}{\partial X_k}  + \alpha_{ij}f.
\end{equation}
We need to compute degree of $\gamma_{ijl}$. Note that $\xi \in (\K_2)_{-\omega}$. So
\[
\frac{b_{ij}}{f^c} \in \left(R_f(\omega_i + \omega_j)\right)_{-\omega}.
\]
It follows that
\begin{equation}\label{bb2}
\deg b_{ij} = c\deg f - \omega + \omega_i + \omega_j
\end{equation}
It can be easily checked that
\[
\ov{b} \in \left(\K_2^\prime\right)_{(c+2)\deg f - \omega}.
\]
So 
$$\gamma \in \left(\K_3^\prime\right)_{(c+2)\deg f - \omega}.$$
It follows that
\begin{equation}\label{bb3}
\deg \gamma_{ijl} = (c-1)\deg f - \omega + \omega_{i} + \omega_j + \omega_l.
\end{equation}
We first consider the case when $c = 1$. Then by equation \ref{bb} we have $\alpha_{ij} = 0$.
Also
$$\deg \gamma_{ijl} = -\omega + \omega_i + \omega_j + \omega _l < 0 \quad \text{if} \ n > 3.$$ 
So if $n > 3$ we get $\gamma_{ijl} = 0$. So $b =0$ and so $\xi = 0$, a contradiction. 

We now consider the case when $n = 3$. Note that $\gamma = \gamma_{123} $ is a constant. Thus
$$ b = \left(\gamma \frac{\partial f}{\partial X_3}, -\gamma \frac{\partial f}{\partial X_2}, \gamma \frac{\partial f}{\partial X_3} \right)^\prime. $$
A direct computation yields $ u = 0$, a contradiction.

We now consider the case when $c \geq 2$.
Notice by equation \ref{bb} we have
\[
\frac{-cb_{ij}}{f^c} = \frac{1}{f^c}\sum_{k < i <j}\gamma_{kij}\frac{\partial f}{\partial X_k} - \frac{1}{f^c} \sum_{i < k <j}\gamma_{ikj}\frac{\partial f}{\partial X_k} +  \frac{1}{f^c}\sum_{i < j < k}\gamma_{ijk}\frac{\partial f}{\partial X_k}  + \frac{\alpha_{ij}}{f^{c-1}}.
\]
Notice
\[
\frac{\gamma_{kij}\partial {f}/\partial X_k}{f^c} = \frac{\partial}{\partial X_k} \left(\frac{\gamma_{kij}/(1-c)}{f^{c-1}} \right)  - \frac{*}{f^{c-1}}
\]
Put
$$ \widetilde{ \gamma_*}  = \frac{1}{c(c-1)}\gamma_*.$$
So we obtain
\[
\frac{b_{ij}}{f^c} = \sum_{k < i < j}\frac{\partial}{\partial X_k}\left(\frac{\widetilde{\gamma_{kij}}}{f^{c-1}}\right)
-  \sum_{i < k < j}\frac{\partial}{\partial X_k}\left(\frac{\widetilde{\gamma_{ikj}}}{f^{c-1}}\right)
+  \sum_{i < j < k}\frac{\partial}{\partial X_k}\left(\frac{\widetilde{\gamma_{ijk}}}{f^{c-1}}\right)
+ \frac{\widetilde{b_{ij}}}{f^{c-1}}.
\]
Set 
$$\delta = \left( \frac{\widetilde{\gamma_{ijl}}}{f^{c-1}} \mid 1 \leq i < j < l \leq n \right) \quad \text{and} \ \widetilde{\xi} =  \left( \frac{\widetilde{b_{ij}}}{f^{c-1}} \mid 1 \leq i < j  \leq n \right). $$
Then
$$\xi = \phi_3(\delta) + \widetilde{\xi}.$$
So we have $u = \phi_2(\xi) = \phi_2(\widetilde{\xi})$. This contradicts choice of $c$.
\end{proof}

\s By Theorem \ref{gEv} we have
\[
H_1(\partial ; R_f) = H_1(\partial ; R_f)_{-\omega} = \frac{Z_1(\partial ; R_f)_{-\omega}}{B_1(\partial ; R_f)_{-\omega}}. 
\]
Let $x \in H_1(\bP; R_f)$ be non-zero. Define
\[
L(x) = \min \{ L(\xi) \mid x = [\xi], \text{where} \ \xi \in Z_1(\bP,R_f)_{-\omega}. \}.
\]
Let $\xi = (\xi_1,\ldots,\xi_n)^\prime \in Z_1(\bP,R_f)_{-\omega}$ be such that $x = [\xi]$. So $\xi \in (\K_1)_{-\omega}$. Thus $\xi_i\in R_f(+\omega_i)_{-\omega}$. So if $\xi \neq 0$ then 
$\xi \in R_f^n \setminus R^n$. It follows that
$L(\xi) \geq 1$. Thus $L(x) \geq 1$.

 We now define a function
\begin{align*}
\widetilde{\theta} \colon H_1(\bP,R_f) &\rt H_1(\partial(f);A) \\
       x &\mapsto \begin{cases} \theta(\xi) & \text{if} \ x\neq 0, x = [\xi], \text{and} \ L(x) = L(\xi),\\ 0 & \text{if} \ x = 0. \end{cases}
\end{align*}
\begin{proposition}\label{indep-1}(with hypotheses as above)
 $\widetilde{\theta(x)}$ is independent of choice of $\xi$.
\end{proposition}
\begin{proof}
Suppose $x = [\xi_1] = [\xi_2]$ is non-zero and $L(x) = L(\xi_1) = L(\xi_2) = i$.
Let $(a_1/f^i,\ldots,a_n/f^i)^\prime$ be normal form of $\xi_1$ and let $(b_1/f^i,\ldots,b_n/f^i)^\prime$ be normal form of $\xi_2$.
It follows that $\xi_1 = \xi_2 + \delta$ where $\delta \in B_1(\partial ; R_f)_{-\omega}$.
By Proposition \ref{prop-l}.1 we get $j = L(\delta) \leq i$. 
Let $(c_1/f^j,\ldots,c_n/f^j)^\prime$ be normal form of $\delta$.
We consider two cases.

Case 1. $j < i$.  Then note that $a_k = b_k + f^{i-j}c_k$ for $k = 1,\ldots,n$. It follows that
\[
\theta(\xi_1)=  [(\ov{a_1},\ldots,\ov{a_n})]  =  [(\ov{b_1},\ldots,\ov{b_n})] = \theta(\xi_2).
\]

Case 2. $j = i$.  Then note that $a_k = b_k + c_k$ for $k = 1,\ldots,n$. It follows that
\[
\theta(\xi_1)=  \theta(\xi_2) + \theta(\delta).
\]
However by Proposition \ref{B}, $\theta(\delta) = 0$. So $\theta(\xi_1)=  \theta(\xi_2)$.
Thus $\widetilde{\theta(x)}$ is independent of 
choice of $\xi$.
\end{proof}

\s We now construct a filtration $\F = \{\F_\nu\}_{\nu\geq 0}$ of $H_1(\bP,R_f)$. Set
\[
\F_\nu = \{ x \in H_1(\bP,R_f)\mid L(x) \leq \nu \}
\]
We prove
\begin{proposition}
\begin{enumerate}[\rm (1)]
\item
$\F_\nu$ is a $K$ subspace of $H_1(\bP;R_f)$.
\item
$\F_{\nu} \supseteq \F_{\nu-1}$ for all $\nu \geq 1$.
\item
$\F_\nu = H_1(\bP;R_f)$ for all $\nu \gg 0$.
\item
$\F_0 = 0$.
\end{enumerate}
\end{proposition}
\begin{proof}
(1) Let $x \in \F_\nu$ and let $\alpha \in K$. Then by * 
\[
L(\alpha x)  \leq L(x) \leq \nu.
\]
So $\alpha x \in \F_\nu$.

Let $x,x^\prime \in \F_\nu$ be non-zero. Let $\xi, \xi^\prime \in Z_1(\partial; R_f)$ be such that $x = [\xi], x^\prime = [\xi^\prime]$ and $L(x) = L(\xi), L(x^\prime) = L(\xi^\prime)$ 
Then $x + x^\prime = [\xi + \xi^\prime]$. It follows that
\[
L(x + x^\prime) \leq L(\xi + \xi^\prime) \leq \max\{ L(\xi) , L(\xi^\prime) \} \leq \nu. 
\]
Note that the second inequality follows from Proposition \ref{prop-l}. Thus $x + x^\prime \in \F_\nu$.

(2) This is clear.

(3) Let  $\mathcal{B} = \{ x_1,\ldots,x_m \} $ be a $K$-basis of  $H_1(\partial; R_f) =  H_1(\partial; R_f)_{-\omega}$.  Let
$$c  = \max\{ L(x_i) \mid i = 1,\ldots, m \}. $$
We claim that
\[
\F_\nu = H_1(\partial; R_f) \quad \text{for all} \ \nu \geq c.
\] 
Fix $\nu \geq c$.
Let $\xi_i \in Z_1(\partial;R_f)_{-\omega}$ be such that $x_i = [\xi_i]$ and $L(x_i) = L(\xi_i)$ for $i =1,\ldots,m$.

Let $u \in H_1(\partial; R_f)$. Say $u  = \sum_{i=1}^{m}\alpha_i x_i$ for some $\alpha_1,\ldots,\alpha_m \in K$. Then 
$u = [ \sum_{i=1}^{m} \alpha_i \xi_i]$.
It follows that
\[
L(u) \leq L(\sum_{i=1}^{m}\alpha_i x_i) \leq \max\{ L(\xi_i) \mid i = 1,\ldots,m \} = c \leq \nu.
\]
Here the second inequality follows from  Proposition \ref{prop-l}. So $u \in \F_\nu$. Thus $\F_\nu = H_1(\partial;R_f)$.

(4) If $x \in H_1(\partial;R_f)$ is non-zero then $L(x) \geq 1$. It follows that $\F_{0} = 0$.
\end{proof}

\s  Let $\G = \bigoplus_{\nu \geq 1} \F_{\nu}/\F_{\nu-1}$. For $\nu \geq 1$ we define
\begin{align*}
\eta_{\nu} \colon \frac{\F_{\nu}}{\F_{\nu-1}} &\rt H_1(\partial(f);A)_{(\nu + 1)\deg f - \omega}\\
                 u &\mapsto \begin{cases}0, & \text{if}  \  u = 0; \\ \widetilde{\theta}(x), & \text{if} \  u = x + \F_{\nu -1} \ \text{is non-zero}.                                                 \end{cases}
\end{align*}

\begin{proposition}\label{indep-2}[with hypotheses as above.]
 $\eta_{\nu}(u)$ is independent of choice of $x$.
 \end{proposition}
\begin{proof}
Suppose $u = x + \F_{\nu-1}   = x^\prime + \F_{\nu -1}$ be non-zero. Then $x = x^\prime + y$ where $y \in \F_{\nu -1}$. As $u \neq 0$ we have that $ x,x^\prime  \in \F_{\nu} \setminus \F_{\nu -1}$.  So $L(x) = L(x^\prime) = \nu$. Say $x = [\xi]$, $x^\prime = [\xi^\prime] $ and $y = [\delta]$ where $\xi,\xi^\prime, \delta \in Z_1(\partial;R_f)$ with 
$L(\xi) = L(\xi^\prime)  = \nu$ and $L(\delta) = L(y)  = k \leq \nu-1$.
So we have that $\xi = \xi^\prime + \delta + \alpha$ where $\alpha \in B_1(\partial; R_f)_{-\omega}$.  Let $L(\alpha) =  r$.  Note $r \leq \nu$.

Let  $(a_1/f^{\nu},\ldots,a_n/f^\nu)^\prime$, $(a_1^\prime/f^\nu, \ldots,a_n^\prime/f^\nu)^\prime$, $(b_1/f^k,\ldots,b_n/f^k)^\prime$  and  \\  $(c_1/f^r,\ldots,c_n/f^r)\prime$ be normal forms of
$\xi, \xi^\prime, \delta$ and $\alpha$ respectively.  So we have
\[
a_j = a_j^\prime + f^{\nu - k}b_j +  f^{\nu - r}c_j \quad \text{for} \ j = 1,\ldots,n.
\]
Case 1:  $r < \nu$. In this case we have that $\ov{a_j} = \ov{a_j^\prime}$ in $A$ for each $j =1,\ldots,n$. So $\theta(\xi) =  \theta(\xi^\prime)$. Thus
$\widetilde{\theta}(x) = \widetilde{\theta}(x^\prime)$.

Case 2: $r = \nu$. In this case notice that $\ov{a_j} = \ov{a_j^\prime} + \ov{c_j}$ in $A$ for each $j =1,\ldots,n$. So $\theta(\xi) =  \theta(\xi^\prime) + \theta(\alpha)$. 
However $\theta(\alpha) = 0$ as $\alpha \in  B_1(\partial;R_f)_{-\omega}$; see Proposition \ref{B}.
 Thus
$\widetilde{\theta}(x) = \widetilde{\theta}(x^\prime)$.
\end{proof} 
 
Note that neither $\theta$ nor $\widetilde{\theta}$ is linear. However we prove the following:

\begin{proposition}\label{linear}(with Notation as above).  For all $\nu \geq 1$,
$\eta_{\nu}$ is $K$-linear.
\end{proposition}
\begin{proof}
Let $u,u^\prime \in \F_{\nu}/\F_{\nu -1}$.
We first show $\eta_\nu(\alpha u) = \alpha \eta_\nu(u)$ for all $\alpha \in K$. We have nothing to show if $\alpha = 0$ or if $u = 0$. So assume $\alpha \neq 0$ and $u \neq 0$.
Say $u = x + \F_{\nu-1}$. Then $\alpha u = \alpha x + \F_{\nu -1}$. As $\widetilde{\theta}(\alpha x) = \alpha \widetilde{\theta}(x)$ we get the result. 

Next we show that $\eta_\nu(u + u^\prime) = \eta_\nu(u) + \eta_\nu(u^\prime)$. We have nothing to show if $u$ or $u^\prime$ is zero. 
Next we consider the case when $u + u^\prime  = 0$. Then $u = -u^\prime$. So $\eta_\nu(u) = - \eta_\nu(u^\prime)$. Thus in this case
\[
\eta_{\nu}(u + u^\prime) = 0 = \eta_{\nu}(u)  + \eta_\nu(u^\prime).
\]
Now consider the case when $u, u^\prime$ are non-zero and $u + u^\prime$ is non-zero. Say $u = x + \F_{\nu-1}$ and $u^\prime = x^\prime + \F_{\nu -1}$. Note that
as $u + u^\prime$ is non-zero $x + x^\prime \in \F_{\nu} \setminus \F_{\nu-1}$.  Let $x = [\xi] $ and $x^\prime = [ \xi^\prime]$ where $\xi,\xi^\prime \in Z_1(\partial;R_f)_{-\omega}$ and $L(\xi) = L(\xi^\prime) = \nu$. Then $x+ x^\prime = [ \xi + \xi^\prime]$. Note that $L(\xi + \xi^\prime) \leq \nu$ by Proposition \ref{prop-l}. But $L(x+x^\prime) = \nu$. So $L(\xi + \xi^\prime) = \nu$.
 Let  $(a_1/f^{\nu},\ldots,a_n/f^\nu)^\prime$, $(a_1^\prime/f^\nu, \ldots,a_n^\prime/f^\nu)^\prime$  be normal forms of
$\xi$ and $ \xi^\prime$ respectively.  Note that $( (a_1+a_1^\prime)/f^\nu,\ldots, (a_n + a_n^\prime)/f^\nu)^\prime$ is normal form of $\xi + \xi^\prime$.
It follows that $\theta(\xi +\xi^\prime) = \theta(\xi) + \theta(\xi^\prime)$. Thus $\widetilde{\theta}(x + x^\prime) = \widetilde{\theta}(x) + \widetilde{\theta}(x^\prime)$. Therefore
\[
\eta_\nu(u + u^\prime) = \eta_\nu(u) + \eta_\nu(u^\prime).
\]
\end{proof}

Finally we have the main result of this section
\begin{proof}[Proof of Theorem 2]
Let $\nu \geq 1$.  By Proposition \ref{linear} we know that $\eta_\nu$ is a linear map of $K$-vector spaces. We now prove that $\eta_\nu$ is injective. 

Suppose if possible $\eta_\nu$ is not injective. Then there exists non-zero $u \in \F_\nu/\F_{\nu-1}$ with $\eta_\nu(u) = 0$.  Say $u = x + \F_{\nu -1}$. Also
let $x = [\xi] $ where $\xi \in Z_1(\partial; R_f)_{-\omega}$ and $L(\xi) = L(x) = \nu$. Let $(a_1/f^\nu,\ldots,a_n/f^\nu)^\prime$ be the normal form of $\xi$. So we have
\[
0 = \eta_\nu(u) =  \widetilde{\theta}(x)  = \theta(\xi) = [(\ov{a_1},\ldots,\ov{a_n})^\prime].
\]
It follows that $(\ov{a_1},\ldots,\ov{a_n})^\prime = \psi_2(\ov{b})$ where $\ov{b} = ( \ov{b_{ij} }\mid 1\leq i < j \leq n )^\prime$.
it follows that for $l =1,\ldots,n$ 
\[
\ov{a_l} = \sum_{i<l} \ov{b_{il}}\frac{\partial f}{\partial X_i}  - \sum_{l>j} \ov{b_{lj}} \frac{\partial f}{\partial X_j}.
\]

It follows that for $l = 1,\ldots,n$   we have the following equation in $R$:
\begin{equation}\label{t2-eq1}
a_l  = \sum_{i<l} b_{il} \frac{\partial f}{\partial X_i}  - \sum_{l>j} b_{lj}\frac{\partial f}{\partial X_j}   + d_lf,
\end{equation}
for some $d_l \in R$. Note that the above equation is of homogeneous elements in $R$.
So we have the following 
\begin{equation}\label{t2-eq2}
\frac{a_l}{f ^\nu}  = \frac{\sum_{i<l} b_{il} \frac{\partial f}{\partial X_i} }{f^{\nu}} - \frac{\sum_{l>j} b_{lj}\frac{\partial f}{\partial X_j}}{f^\nu}   +  \frac{d_l}{f^{\nu-1}}.
\end{equation}
We consider two cases:

Case 1: $\nu \geq 2$.

Set $\widetilde{b_{ij}} =  -b_{il}/(c-1)$. Then note that 
\[
\frac{ b_{il} \frac{\partial f}{\partial X_i}}{f^\nu} = \frac{\partial}{\partial X_i} \left( \frac{\widetilde{b_{il}}}{f^{\nu-1}}\right)    - \frac{*}{f^{\nu-1}}.
\]

By \ref{t2-eq2} we have for $l = 1,\ldots,n$,
\[
\frac{a_l}{f^\nu} = \sum_{i<l}\frac{\partial}{\partial X_i} \left( \frac{\widetilde{b_{il}}}{f^{\nu-1}}\right)    -    
\sum_{l<j}\frac{\partial}{\partial X_j} \left( \frac{\widetilde{b_{lj}}}{f^{\nu-1}}\right)    + \frac{c_l}{f^{\nu-1}}.
\]
Put $\xi^\prime = (c_1/f^{\nu-1},\ldots,c_n/f^{\nu-1})^\prime$ and $\delta = ( \widetilde{b_{ij}}/f^{\nu-1} \mid 1 \leq i < j \leq n)$. Then we have 
\[
\xi = \phi_2(\delta) + \xi^\prime.
\]
 So we have $x = [\xi] = [\xi^\prime]$. This yields $L(x) \leq L(\xi^\prime) \leq \nu -1$. This is a contradiction.

Case 2: $\nu = 1$.

Note that $\xi \in (\K_1)_{-\omega}$. Thus for  $l = 1,\ldots,n$ we have 
\[
\frac{a_l}{f}  \in (R_f(\omega_l))_{-\omega} . 
\] 
 It follows that
 \[
 \deg a_l = \deg f - \sum_{k\neq l}\omega_k.
 \]
 Also note that $\deg \partial f/ \partial X_i = \deg f - \omega_i$.  By  comparing degrees in \ref{t2-eq1} we get
 $a_l = 0$ for all $l$. Thus $\xi = 0$. So $x = 0.$ Therefore $u = 0$ a contradiction.
\end{proof}

\section{Example 0.1}
Let $R = K[X_1,\ldots, X_n]$ and let $f = X_1^2+ \ldots+ X_{n-1}^2 + X_n^m$ with $m \geq 2$. Set $A = R/(f)$.
In this section we compute  $H_1(\partial;H^1_{(f)}(R))$. 

\s We give $\omega_i = \deg X_i = m$ for $i = 1,\ldots, n-1$ and $ \omega_n = \deg X_n = 2$. Note that $f$ is a homogeneous polynomial in $R$
of degree $2m$. Also note that $\omega = \sum_{k=1}^{n} \omega_k = (n-1)m + 2$.

\s  First note that the Jacaobian ideal $J$ of $f$ is primary to the unique graded maximal ideal of $R$. It follows that $A$ is an isolated singularity. Note $J = (X_1,\ldots,X_{n-1}, X_n^{m-1})$. Let $H_i(J;A)$ be the $i^{th}$-Koszul homology of $A$ \wrt \ $J$.

\begin{proposition}\label{HK}
The Hilbert series, $P(t)$, of $H_1(J;A)$ is 
\[
P(t) = \sum_{k = 0}^{m-2}t^{2m + 2j}.
\]
\end{proposition}
\begin{proof}
It is easily verified that $X_1,\ldots,X_{m-1}$ is an $A$-regular sequence. Set 
$$B = A/(X_1,\ldots,X_{m-1})A = \frac{K[X_n]}{(X_n^m)} = K \oplus K X_n \oplus X_n^2 \oplus \cdots \oplus KX_n^{m-1}.$$

 Note that we have an exact sequence
 \[
 0 \rt H_1(J;A) \rt B(-2(m-1)) \xrightarrow{X_n^{m-1}} B.
 \]
 It follows that $H_1(J;A) = X_nB(-2(m-1))$. The result follows
\end{proof}

\s By Theorem 2 there exists 
 a filtration $\F = \{ \F_\nu \}_{\nu\geq 0}$ consisting of $K$-subspaces of $H_{1}(\bP; H^1_{(f)}(R))$ with $\F_\nu = H^{n-1}(\bP; H^1_{(f)}(R))$ for $\nu \gg 0$, $\F_\nu \supseteq F_{\nu-1}$ and $\F_0 = 0$ and injective $K$-linear maps
\[
\eta_\nu \colon \frac{\F_\nu}{\F_{\nu-1}} \longrightarrow H_{1}\left(\partial(f); A\right)_{(\nu+1)\deg f - \omega}.
\]
Notice
\[
(\nu+1)\deg f - \omega = (\nu+1)2m - (n-1)m - 2 = (2\nu - n + 3)m - 2.
\]
If $\eta_\nu \neq 0$ then by Proposition \ref{HK} it follows that
\[
(2\nu - n + 3)m - 2 = 2m + 2j  \quad \text{for some} \ j = 0,\ldots,m-2.
\]
So we obtain
\begin{equation}\label{comput}
2\nu m = (n-1)m + 2(j+1)
\end{equation}
It follows that $m$ divides $2(j+1)$. As $2(j+1) \leq 2m-2$ it follows that $2(j+1) = m$. Thus $m$ is even.

\s Say $m = 2r$. Then by equation (\ref{comput}) we have
\[
2\nu r = (n-1)r + r.
\]
So $\nu = n/2$. It follows that $n$ is even. Furthermore note that $\eta_\nu = 0$ for $\nu \neq n/2$ and that if $\nu = n/2$ then by \ref{HK}, $\dim \F_{n/2}/\F_{n/2 -1} \leq 1$. It follows that in this case $\dim H_1(\partial; H^1_{(f)}(R))\leq 1$.

\s In conclusion we have
\begin{enumerate}
\item
if $m$ is odd then
$H^{n-1}(\partial; H^1_{(f)}(R)) = 0$.
\item
if $m$ is even then
\begin{enumerate}
\item
If $n$ is odd then $H^{n-1}(\partial; H^1_{(f)}(R)) = 0$.
\item
If $n$ is even then $\dim_K H^{n-1}(\partial; H^1_{(f)}(R)) \leq 1$.
\end{enumerate}
\end{enumerate}
This proves Example 0.1.

\end{document}